\documentclass[12pt,twoside]{amsart}

\usepackage{amsmath}
\usepackage{amssymb}
\usepackage{amscd}
\usepackage{graphics}
\usepackage{graphicx}

\renewcommand{\bar}{\overline}

\newcommand{\AAA}{\mathbb{A}}

\newcommand{\FF}{\mathbb{F}}

\newcommand{\PP}{\mathbb{P}}
\newcommand{\QQ}{\mathbb{Q}}
\newcommand{\RR}{\mathbb{R}}

\newcommand{\ZZ}{\mathbb{Z}}

\newcommand{\Qbar}{\bar{\QQ}}

\newcommand{\Fp}{\FF_p}

\newcommand{\calM}{{\mathcal M}}

\newcommand{\calO}{{\mathcal O}}
\newcommand{\calP}{{\mathcal P}}

\DeclareMathOperator{\PGL}{PGL}
\DeclareMathOperator{\Preper}{Preper}

\DeclareMathOperator{\Res}{Res}

\newcommand{\dsps}{\displaystyle}

\theoremstyle{plain}
\newtheorem{thm}{Theorem}[section]

\newtheorem{lemma}[thm]{Lemma}
\newtheorem{conj}{Conjecture}

\newtheorem*{thmblank}{Theorem}

\theoremstyle{definition}
\newtheorem{defin}[thm]{Definition}
\newtheorem{alg}[thm]{Algorithm}

\theoremstyle{remark}
\newtheorem{remark}[thm]{Remark}

\numberwithin{equation}{section}

\title[Dynamics of quadratic maps]
{Small dynamical heights for quadratic polynomials
and rational functions}
\date{November 30, 2013; revised June 12, 2014}

\author[Benedetto]{Robert~L. Benedetto}
\address[Benedetto, Chen, White]
	{Amherst College \\ Amherst, MA 01002}
\email[Benedetto]{rlb@math.amherst.edu}
\author[Chen]{Ruqian Chen}
\email[Chen]{rchen13@amherst.edu}
\author[Hyde]{Trevor Hyde}
\email[Hyde]{tghyde@umich.edu}
\author[Kovacheva]{Yordanka Kovacheva}
\email[Kovacheva]{ykovach@math.uchicago.edu}
\author[White]{Colin White}
\email[White]{crwhite14@amherst.edu}
\address[Hyde]{University of Michigan\\ Ann Arbor, MI 48109}
\address[Kovacheva]{University of Chicago\\ Chicago, IL 60637}

\begin{document}

\newcounter{bean}
\newcounter{sheep}

\begin{abstract}
Let $\phi\in \QQ(z)$ be a polynomial or rational function of degree $2$.
A special case of Morton and Silverman's
Dynamical Uniform Boundedness Conjecture states that
the number of rational preperiodic points of $\phi$
is bounded above by an absolute constant.
A related conjecture of Silverman states that the canonical height
$\hat{h}_{\phi}(x)$ of a non-preperiodic rational point $x$ is
bounded below by a uniform multiple of the height of $\phi$ itself.
We provide support for these conjectures by computing the set
of preperiodic and small height rational points
for a set of degree~2
maps far beyond the range of previous searches.
\end{abstract}

\maketitle

In this paper, we consider the dynamics of a rational function
$\phi(z)\in\QQ(z)$ acting on $\PP^1(\QQ)$.
The degree of $\phi=f/g$ is $\deg\phi:=\max\{\deg f, \deg g\}$,
where $f,g\in\QQ[z]$ have no common factors.
Define $\phi^0(z)=z$, and for
every $n\geq 1$, let $\phi^n(z)=\phi\circ\phi^{n-1}(z)$; that is,
$\phi^n$ is the $n$-th iterate of $\phi$ under composition.
In this context,
the automorphism group $\PGL(2,\QQ)$ of $\PP^1(\QQ)$
acts on $\QQ(z)$ by conjugation.

The \emph{forward orbit} of a point $x\in\PP^1(\QQ)$
is the set of iterates
$$\calO(x) = \calO_{\phi}(x):=\{\phi^n(x) : n\geq 0\}.$$
The point $x$ is said to be {\em periodic}
under $\phi$ if there is an integer $n\geq 1$ such that
$\phi^n(x)=x$.  In that case, we say $x$ is $n$-periodic,
we call the orbit $\calO(x)$ an $n$-cycle,
and we call $n$ the \emph{period} of $x$, or of the cycle.
The smallest period $n\geq 1$ of a periodic point $x$
is called the \emph{minimal} period of $x$, or of the cycle.
More generally, $x$ is {\em preperiodic} under
$\phi$ if there are integers $n>m\geq 0$ such that
$\phi^n(x)=\phi^m(x)$.  Equivalently, $\phi^m(x)$ is
periodic for some $m\geq 0$;
also equivalently, the forward orbit $\calO(x)$ is finite.
We denote the set of preperiodic points of $\phi$ in
$\PP^1(\QQ)$ by $\Preper(\phi,\QQ)$.

In 1950, using the theory of arithmetic heights,
Northcott \cite{Nor} proved that if $\deg\phi\geq 2$, then
$\phi$ has only finitely many preperiodic points in $\PP^1(\QQ)$.
(In fact, he proved a far more general finiteness result,
for morphisms of $\PP^N$ over any number field.)
In 1994,
Morton and Silverman proposed
a dynamical Uniform Boundedness Conjecture
\cite{MS1,MS2}; for $\phi\in\QQ(z)$ acting on $\PP^1(\QQ)$, it
says the following.

\begin{conj}[Morton-Silverman, 1994]
\label{conj:ubc}
For any $d\geq 2$, there is a constant $M=M(d)$ such that
for any $\phi\in\QQ(z)$ of degree $d$,
$$\#\Preper(\phi,\QQ)\leq M.$$
\end{conj}

Only partial results towards
Conjecture~\ref{conj:ubc} have been proven, including
non-uniform bounds of various strengths, as well as conditions
under which certain preperiodic orbit structures are
possible or impossible.
See, for example
\cite{Ben9,CaGol,FPS,Man,Mor1,Mor2,MS1,MS2,Nar,Pez1,Poo,Zieve},
as well as \cite[Section~4.2]{Sil}.

Poonen \cite{Poo} later stated a sharper version of
the conjecture for the special case of quadratic polymials over $\QQ$:

\begin{conj}[Poonen, 1998]
\label{conj:poo}
Let $\phi\in\QQ[z]$ be a polynomial of degree $2$.
Then $\#\Preper(\phi,\QQ)\leq 9$.
\end{conj}

If true, Conjecture~\ref{conj:poo} is sharp; for example,
$z^2-29/16$ and $z^2 - 21/16$ each have exactly $9$ rational
preperiodic points, including the point at $\infty$.
However, even though it is the simplest case of Conjecture~\ref{conj:ubc},
a proof of Conjecture~\ref{conj:poo} seems to be very far off
at this time.

As little as we know about the uniform boundedness conjecture
for quadratic polynomials, we know even less about rational functions.
In her Ph.D.\ thesis, Manes \cite{Man} made the first systematic attack
on preperiodic points of quadratic rational functions,
including a conjecture that 
$\#\Preper(\phi,\QQ)\leq 12$ when $\phi(z)\in\QQ(z)$
has $\deg\phi=2$.  In this paper, we give examples with
$14$ rational preperiodic points, showing
that  Manes' conjecture is false.
On the one hand, we found a single map with a rational
$7$-cycle, along with the immediate preimages of all seven points;
see equation~\eqref{eq:7per}.
On the other hand, we found many maps with a rational point $x$
whose sixth iterate $\phi^6(x)$ is $2$-periodic;
the immediate preimages of all those preperiodic points again
give a total of 14 points.
We also found a single map with a rational point $x$
for which $\phi^5(x)$ is $3$-periodic, again giving a total
of 14 points.
See Table~\ref{tab:quadratlen8} for examples.
It would appear that there are only finitely many maps with a
7-cycle or with a 3-periodic cycle with a tail of length 5,
while there seem to be infinitely many with a 2-periodic cycle
having a tail of length 6; for the moment, however,
those finiteness questions remain open.
Meanwhile, Blanc, Canci, and Elkies, in their study of
the space of degree two
rational functions with a rational point of period $6$,
have recently announced an infinite family of quadratic
maps with $14$ $\QQ$-rational preperiodic points;
see \cite[Lemma 4.7]{BC}.
Their family uses two separate orbits: a $6$-cycle and a fixed point,
together with the preimages of all seven periodic points.

Besides its preperiodic orbits,
any rational function $\phi\in\QQ(z)$ of degree $d\geq 2$
has an associated \emph{canonical height}.
The canonical height is a function
$\hat{h}_{\phi}:\PP^1(\Qbar)\rightarrow [0,\infty)$
satisfying the functional
equation $\hat{h}_{\phi}(\phi(z)) = d\cdot \hat{h}_{\phi}(z)$,
and it has the property that
$\hat{h}_\phi(x)=0$ if and only if
$x$ is a preperiodic point of $\phi$;
see Section~\ref{sect:background}.
For a non-preperiodic point $y$, on the other hand,
$\hat{h}_{\phi}(y)$ measures how fast the standard Weil height
$h(\phi^n(y))$ of the iterates of $y$ increases with $n$.
By an analogy with Lang's height lower bound
conjecture for elliptic curves, Silverman has asked
how small $\hat{h}_{\phi}(y)$ can be for non-preperiodic points $y$.
More precisely, considering $\phi$ as a point in
the appropriate moduli space, and defining $h(\phi)$
to be the Weil height of that point, he stated the
following conjecture;
see \cite[Conjecture~4.98]{Sil} for a more general version.

\begin{conj}[Silverman, 2007]
\label{conj:ht}
Let $d\geq 2$.  Then there is a positive constant $M'=M'(d)>0$
such that for any $\phi\in\QQ(z)$ of degree $d$ and any
point $x\in\PP^1(\QQ)$ that is not preperiodic for $\phi$, we have
$\hat{h}_{\phi}(x)\geq M' h(\phi)$.
\end{conj}



Conjecture~\ref{conj:ht}
essentially says that the height of a non-preperiodic rational
point must start to grow rapidly
within a bounded number of iterations.
Some theoretical evidence for Conjecture~\ref{conj:ht}
appears in \cite{Bak,Ing},
and computational evidence for polynomials of degree $d=2,3$
appears in \cite{REU07,DFK,Gil}.
The smallest known value of
$\hat{h}_{\phi}(x)/h(\phi)$
when $\phi$ is a polynomial of degree $2$
occurs for $x=\frac{7}{12}$ and
$\phi(z) = z^2 - \frac{181}{144}$.
The first few iterates of this pair $(x,\phi)$,
first discovered in \cite{Gil}, are
$$\frac{7}{12} \mapsto
-\frac{11}{12} \mapsto
\frac{5}{12} \mapsto
-\frac{13}{12} \mapsto
-\frac{1}{12} \mapsto
-\frac{5}{4} \mapsto
\frac{11}{36} \mapsto
-\frac{377}{324} \mapsto
\frac{2445}{26244} \mapsto \cdots.
$$
The small canonical height ratio
$\hat{h}_{\phi}(\frac{7}{12})/h(\phi)\approx .0066$ 
makes precise the observation that although the numerators
and denominators of the iterates eventually explode in size,
it takes several iterations for the explosion to get underway.

In this paper, we investigate quadratic polynomials and
rational functions with coefficients in $\QQ$, looking for
rational points that either are preperiodic or have small
canonical height.  More precisely, we search for
pairs $(x,\phi)$ for which either $x$ is preperiodic under $\phi$
or the ratio $\hat{h}_{\phi}(x)/h(\phi)$ is positive but
especially small.

Past computational investigations of this type
(such as those in \cite{DFK,Gil} for quadratic polynomials,
or \cite{REU07} for cubic polynomials)
have started with the map $\phi$ and then computed the
full set $\Preper(\phi,\QQ)$ of rational preperiodic points,
or the full set of rational points of small canonical height.
This is a slow process, because for any given $\phi$, the
region in $\PP^1(\QQ)$ that must be exhaustively searched
is usually quite large, and the overwhelming majority of points
in the region turn out to be false alarms.
However, since we are looking for only a \emph{single} point $x$ 
with an interesting forward orbit,
we can start with the point $x$ and \emph{then} search for
maps $\phi$ that give interesting orbits for $x$.
This strategy ends up testing far fewer pairs $(x,\phi)$ that
never really had a chance of being preperiodic or having
small height ratio, and hence
we can push our computations much further.

Our computations provide further evidence for
Conjectures~\ref{conj:ubc}, \ref{conj:poo}, and~\ref{conj:ht}.
First, in spite of
our very large search region, we found no
quadratic polynomials with any $\QQ$-rational
preperiodic structures not already
observed and classified in \cite{Poo}, providing support for
Conjecture~\ref{conj:ubc}, or more precisely, for Conjecture~\ref{conj:poo}.
Second, we also found
a number of new pairs $(x,\phi)$ with small canonical height ratio
for $\phi$ a quadratic polynomial,
but $(7/2, z^2-181/144)$ remains the record-holder.
This supports Conjecture~\ref{conj:ht}, as even our much larger
search region turned up no points breaking the previously
existing record.
In particular, we are led to propose the following.

\begin{conj}
\label{conj:qpolylang}
Let
$$C_{\textup{poly},2}:=\frac{\hat{h}_{z^2-181/144}
\big(7/12\big)}{h\big(181/144\big)}
= \frac{.03433\ldots}{\log 181} \approx .00660.$$
For any polynomial $\phi\in\QQ[z]$ with $\deg\phi=2$,
let $h(\phi)=h(c)$, where $c\in\QQ$ is the unique rational
number such that $\phi$ is conjugate to $z\mapsto z^2+c$.
Then $\hat{h}_{\phi}(x)\geq C_{\textup{poly},2} \cdot h(c)$
for any $x\in\PP^1(\QQ)$.
\end{conj}

Third, regarding quadratic rational functions, which had
not been previously studied at this level of generality,
we found several new preperiodic orbit structures.
Most notably, we found
many degree two maps with 14 $\QQ$-rational preperiodic points,
including one with a $\QQ$-rational periodic cycle of period~7:
\begin{equation}
\label{eq:7per}
\phi(z) = \frac{4655z^2 - 4826z + 171}{4655z^2 - 8071z + 798},
\end{equation}
with periodic cycle
$$\infty \mapsto 1
\mapsto 0
\mapsto \frac{3}{14}
\mapsto \frac{19}{21}
\mapsto \frac{1}{7}
\mapsto \frac{57}{35}
\mapsto \infty.$$
Each of the seven rational points above
also has exactly one rational preimage
outside the cycle.  Those points are:
$$\frac{2}{19},
\frac{57}{295},
\frac{9}{245},
\frac{563}{665},
-\frac{29}{5},
\frac{3}{190},
\frac{27}{10},$$
respectively, and $\Preper(\phi,\QQ)$
consists of precisely these fourteen points.

Fourth, we found pairs $(x,\phi)\in\PP^1(\QQ)\times\QQ(z)$ with
$\deg(\phi)=2$ and with canonical height ratio
$\hat{h}_\phi(x)/h(x)$ much smaller than $.0066$, the
record for quadratic \emph{polynomials}.  But again,
the pairs $(x,\phi)$ of especially small height ratio, as well
as the first maps with $14$ rational preperiodic points,
turned up early in our large search.
This support for Conjectures~\ref{conj:ubc}
and~\ref{conj:ht} inspired the following more
specific statements for quadratic rational maps.

\begin{conj}
\label{conj:qrat}
Let $\psi(z)=(10z^2 - 7z - 3)/(10z^2 + 37z + 9)$, and let
$$C_{\textup{rat},2}:=\hat{h}_{\psi}(\infty)/h(\psi)
\approx .000466.$$
For any rational function $\phi\in\QQ(z)$ with $\deg\phi=2$,
and any $x\in\PP^1(\QQ)$
\begin{enumerate}
\item $\#\Preper(\phi,\QQ)\leq 14$.
\item If $x$ is preperiodic, then $\#\calO_{\phi}(x)\leq 8$.
\item If $x$ is periodic, then $\#\calO_{\phi}(x)\leq 7$.
\item If $x$ is not preperiodic,
then $\hat{h}_{\phi}(x)\geq C_{\textup{rat},2} \cdot h(\phi)$.
\end{enumerate}
\end{conj}

The outline of the paper is as follows.  In Section~\ref{sect:background}
we review some background: heights, canonical heights,
multipliers of periodic points, valuations, and good reduction.
In Section~\ref{sect:qpolys} we recall and then sharpen some
known facts about the dynamics of quadratic polynomials over~$\QQ$,
and we describe our search algorithm.  We then summarize 
and discuss our data from that search.
Finally, in Section~\ref{sect:qrats}, we do a similar analysis
for quadratic rational functions over $\QQ$.  We also state
and prove the following result suggested by our data;
see Theorem~\ref{thm:52ell} for a more precise version.

\begin{thmblank}
Let $X_{5,2}$ be the parameter space of all pairs
$(x,\phi)$ with $x\in\PP^1$ and $\phi$ a rational function
of degree~$2$, up to coordinate change, for which the
forward orbit of $x$ consists of five strictly preperiodic
points followed by a periodic cycle of period~$2$.
Then $X_{5,2}$ is birational over $\QQ$ to an elliptic surface
of positive rank over $\QQ(t)$
with infinitely many $\QQ$-rational points.
\end{thmblank}

\section{Background}
\label{sect:background}

The standard (Weil) {\em height function} on $\QQ$ is the function
$h:\PP^1(\QQ)\to\RR$ given by
$$h(x) := \log\max\{|m|,|n|\},$$ if we write
$x=m/n$ in lowest terms and write $\infty$ as $1/0$.
It has a well-known extension to a function $h:\PP^1(\overline{\QQ})\to\RR$,
although that extension will not be of much concern to us here.
The height function satisfies two important properties.
First, 
for any $\phi(z)\in\QQ(z)$,
there is a constant $C=C(\phi)$ such that
$$\big|h(\phi(x)) - d\cdot h(x)\big| \leq C
\qquad \text{for all } x\in\PP^1(\Qbar),$$
where $d=\deg\phi$.
Second, for any bound $B\in\RR$,
\begin{equation}
\label{eq:htfin}
\{x\in\PP^1(\QQ) : h(x) \leq B\} \quad \text{ is a  finite set}.
\end{equation}

For any fixed $\phi\in\QQ(z)$ of degree $d\geq 2$,
the {\em canonical height} function
$\hat{h}_{\phi}:\PP^1(\Qbar)\to\RR$  for $\phi$ is given by
$$\hat{h}_{\phi}(x) := \lim_{n\to\infty} d^{-n} h(\phi^n(x)),$$
and it satisfies the functional equation
\begin{equation}
\label{eq:hfeqn}
\hat{h}_{\phi}(\phi(x)) = d\cdot \hat{h}_{\phi}(x)
\qquad \text{for all } x\in\PP^1(\Qbar).
\end{equation}
In addition,
there is a constant $C'=C'(\phi)$ such that
\begin{equation}
\label{eq:hclose2}
\big|\hat{h}_{\phi}(x) - h(x)\big| \leq C'
\qquad \text{for all } x\in\PP^1(\Qbar).
\end{equation}
Northcott's Theorem \cite{Nor} that $\#\Preper(\phi,\QQ)<\infty$
is immediate from
properties \eqref{eq:htfin}, \eqref{eq:hfeqn}, and \eqref{eq:hclose2},
since they imply that for any $x\in\PP^1(\QQ)$,
$\hat{h}(x)=0$ if and only if $x$ is preperiodic under $\phi$.
(In fact, Northcott proved this equivalence for any $x\in\PP^1(\Qbar)$).
Meanwhile, our canonical height computations will require
the following result.

\begin{lemma}
\label{lem:htfunc}
Let $\phi=f/g\in\QQ(z)$, where
$f,g$ are relatively prime polynomials in $\ZZ[z]$
with $d:=\max\{\deg f, \deg g\}\geq 2$.
Let $R=\Res(f,g)\in\ZZ$ be the resultant of $f$ and $g$, and let
$$D:= \min_{t\in\RR\cup\{\infty\}}
\frac{\max\{ |f(t)|,|g(t)|\}}{\max\{ |t|^d , 1\}}.$$
Then $D>0$, and for all $x\in\PP^1(\QQ)$
and all integers $i\geq 0$,
$$\hat{h}_{\phi}(x) \geq
d^{-i}\bigg[ h\big(\phi^i(x)\big)
- \frac{1}{d-1}\log\bigg(\frac{|R|}{D}\bigg) \bigg].$$
\end{lemma}

\begin{proof}
The function
$F(t)=\max\{ |f(t)|,|g(t)|\} / \max\{ |t|^d , 1\}$ is defined
at $\infty$ because $\max\{\deg f, \deg g\}=d$.  Moreover, $F$
is real-valued, positive, and continuous on the compact set
$\RR\cup\{\infty\}$, and hence the minimum $D$ is indeed
both attained and positive.
As shown in the proof of Lemma~III.3$'$.(b) of \cite{ST},
$$h(\phi(x)) \geq d h(x) + \log\bigg(\frac{F(x)}{|R|}\bigg)$$
for all $x\in \PP^1(\QQ)$.
Bounding $F(x)$ by $D$, and applying $\phi$
repeatedly to $\phi^i(x)$, we have
$$h\big(\phi^{i+m}(x)\big) \geq d^m h\big(\phi^i(x)\big) +
(1+d+d^2+ \cdots + d^{m-1}) \log\bigg(\frac{D}{|R|}\bigg)$$
for all $x\in\QQ$ and all $i,m\geq 0$.
Dividing by $d^{i+m}$ and taking the limit as $m\to\infty$ gives
the desired inequality.
\end{proof}

For more background on heights and canonical heights,
see \cite[Section~B.2]{HS}, \cite[Chapter~3]{Lang}, 
or \cite[Chapter~3]{Sil}.

The following notion of good reduction of a
dynamical system, which we state here only over $\QQ$,
was first proposed in \cite{MS1}.

\begin{defin}
\label{def:goodred}
Let $\phi(z)=f(z)/g(z)\in\QQ(z)$ be a rational function,
where $f,g\in\ZZ[z]$ have no common factors in $\ZZ[z]$.
Let $d=\deg\phi = \max\{\deg f,\deg g\}$,
and let $p$ be a prime number.
Let $\bar{f},\bar{g}\in\Fp[z]$ be the reductions of
$f$ and $g$ modulo $p$.  If $\deg(\bar{f}/\bar{g})=d$,
we say that $\phi$ has {\em good reduction} at $p$;
otherwise, we say $\phi$ has {\em bad reduction} at $p$.
\end{defin}

Of course, $\deg(\bar{f}/\bar{g})\leq d$ always; but the
degree can drop if either $\max\{\deg\bar{f},\deg\bar{g}\}<d$
or $\bar{f}$ and $\bar{g}$ are no longer relatively prime.
If $\phi=a_d z^d + \cdots +a_0$ is a polynomial, then
$\phi$ has good
reduction at $p$ if and only if $v_p(a_i)\geq 0$ for all $i$,
and $v_p(a_d)=0$; see \cite[Example~4.2]{MS2}.
Here, 
$v_p(x)$ denotes the $p$-adic valuation of $x\in\QQ^{\times}$,
given by $v_p(p^r a/b) = r$ where $a,b,r\in\ZZ$ and $p\nmid ab$.

Finally,
if $\phi\in\QQ(z)$ is a rational function, and $x\in\Qbar$
is a periodic point of $\phi$ of minimal period $n$,
the \emph{multiplier} of $x$ is $(\phi^n)'(x)$.  The multiplier
is invariant under coordinate change, and therefore one can compute
the multiplier of a periodic point at $x=\infty$ by changing coordinates
to move it elsewhere.
We will need multipliers to discuss rational functions
in Section~\ref{sect:qrats}.

\section{Quadratic Polynomials}
\label{sect:qpolys}

It is well known that (except in characteristic 2),
any quadratic polynomial is conjugate over the base field
to a unique one of the form $\phi_c(z):=z^2+c$.
(Scaling guarantees the polynomial is monic, and an
appropriate translation to complete the square eliminates
the linear term.)  Thus, the moduli space of quadratic
polynomials up to conjugacy is $\AAA^1$, where the parameter
$c\in\AAA^1$ corresponds to $\phi_c$.
For the purposes of Conjecture~\ref{conj:ht}, then,
the height of $\phi_c$ itself is
$$h(\phi_c) := h(c).$$
In addition, let us denote the canonical height associated to
$\phi_c$ as simply $\hat{h}_c$.

However, we are really interested in \emph{pairs}
$(x,c)$ for which the point $x\in\PP^1$ either is preperiodic
with a long forward orbit under $\phi_c$, or else has very
small canonical height under $\phi_c$, as compared with $h(c)$.
The appropriate moduli space of such pairs is therefore
$\PP^1\times\AAA^1$.  If $x=\infty$, then
$x$ is simply a fixed point with
canonical height $0$.  Thus, we may restrict ourselves
to the subspace $\AAA^2\cong \AAA^1\times\AAA^1\subseteq\PP^1\times\AAA^1$.
When working over $\QQ$, we can restrict ourselves even further.
The following lemma is well known, but we include the short proof
for the convenience of the reader.

\begin{lemma}
\label{lem:quadpoly}
Let $x,c\in\QQ$, and suppose that $x$ is preperiodic
under $\phi_c(z) = z^2+c$.  Then writing $x=m/n$ in lowest
terms, we must have $c=k/n^2$, where $k\in\ZZ$ satisfies
\begin{enumerate}
\item  $\dsps k\equiv -m^2 \pmod{n}$,
\item  the integers $n$ and $(k+m^2)/n$ are relatively prime,
\item  $k\leq n^2/4$, and
\item  $\dsps |x|\in
\begin{cases}
[0,2] & \text{ if } c\geq -2,
\\
[\sqrt{-c-B},B] & \text{ if } c < -2,
\end{cases} \quad$
where $B:=(1+\sqrt{1-4c})/2$.
\end{enumerate}
\end{lemma}

\begin{proof}
Write $c=k/N$, where $k,N\in\ZZ$ are relatively prime, and $N\geq 1$.
If $N\neq n^2$, then there is a prime $p$ such that
$v_p(N)\neq 2 v_p(n)$.
Let $r:=v_p(n)$ and $s:= v_p(N)$.  Then
\begin{equation}
\label{eq:mnp}
\phi_c(x) = \frac{m^2}{n^2} + \frac{k}{N}
= \frac{m^2 N + kn^2}{n^2 N}.
\end{equation}
The denominator of \eqref{eq:mnp} has $v_p(n^2 N) = 2r+s$.
If $s>2r$, then the numerator has $v_p(m^2 N + kn^2) = 2r$,
and hence $v_p(\phi_c(x))=-s$.
On the other hand,
if $s<2r$, then the numerator has $v_p(m^2 N + kn^2) = s$,
and hence $v_p(\phi_c(x))=-2r$.
Either way, then, $v_p(\phi_c(x))<-r=v_p(x)$.
Thus, $v_p(\phi_c^i(x))$ will strictly decrease with $i$, contradicting
the hypothesis that $x$ is preperiodic.
Therefore, $c=k/n^2$, with $k\in\ZZ$ relatively prime to $n$.

Applying the previous paragraph to \emph{any} preperiodic point,
not just $x$, we see that all preperiodic points of $\phi$ in $\QQ$
have denominator exactly $n$ when written in lowest terms.
In particular, $\phi_c(x) = (k+m^2)/n^2$ is preperiodic.
Hence, $n|(k+m^2)$, and $(k+m^2)/n$ must be relatively prime to $n$,
giving us statements~(a) and~(b).

Next, if $c>1/4$, then for any $y\in\RR$, we have
$$\phi(y) -y = y^2 - y +c > y^2 - y + \frac{1}{4}
= \Big(y-\frac{1}{2}\Big)^2 \geq 0.$$
Thus, $\{\phi_c^n(x)\}_{n\geq 0}$ is a strictly increasing
sequence, contradicting the hypothesis that $x$ is preperiodic.
Hence, $c\leq 1/4$, giving us statement~(c).

Finally, if $c\geq -2$, then for any $y\in\RR$ with $|y|>2$,
$$\phi(y) - |y| = |y|^2 - |y| + c
\geq |y|^2 - |y| -2 = (|y|-2)(|y|+1)>0.$$
Just as in the previous paragraph, then,
$\{\phi_c^n(x)\}_{n\geq 0}$ is strictly increasing
if $|x|>2$, contradicting the preperiodicity hypothesis.
Similarly, if $c< -2$, then noting that
$A=(1-\sqrt{1-4c})/2<-1$ and $B=(1+\sqrt{1-4c})/2>2$
are the two (real) roots of $\phi(z)-z=0$,
we have, for any $y\in\RR$ with $|y|>B$,
$$\phi(y) - |y| = |y|^2 - |y| + c
=(|y|-A)(|y|-B)>0.$$
Meanwhile, $\phi_c^{-1}([-B,B])=[-B,-\sqrt{-c-B}]\cup
[\sqrt{-c-B},B]$, and hence for $x\in\QQ$
to be preperiodic, we must have $|x|\in [\sqrt{-c-B},B]$.
\end{proof}

We can improve the bound
of Lemma~\ref{lem:quadpoly}.(c) with a minor extra assumption,
as follows.

\begin{lemma}
\label{lem:quadpoly2}
With notation as in Lemma~\ref{lem:quadpoly},
let $\calO_c(x)$ denote the forward orbit of $x$ under $\phi_c$.
If $\#\calO_c(x) \geq 4$, then $k < -3n^2/4$.
\end{lemma}

\begin{proof}
Given Lemma~\ref{lem:quadpoly}.(c),
we need to show that if $c\in [-3/4,1/4]$, then
the preperiodic point $x$ has $\#\calO_c(x) \leq 3$.

First of all, using Lemma~\ref{lem:quadpoly},
it is easy to see that
$$\Preper(\phi_{1/4},\QQ) = \Big\{\pm \frac{1}{2} \Big\},
\quad\text{and}\quad
\Preper(\phi_{-3/4},\QQ) = \Big\{\pm \frac{1}{2}, \pm \frac{3}{2}\Big\}.$$
In both cases, it follows immediately that
each preperiodic point
has forward orbit of length at most $2$.

Second of all, for any $c\in (-3/4,1/4)$, the dynamics of $\phi_c$
on $\RR$ has two fixed points, at $a<b\in\RR$, and every point in
$\RR$ is either attracted to $\infty$ under iteration,
or attracted to $a$ under iteration,
or equal to $\pm b$; and $\phi_c(\pm b) = b$.
This well known fact is easy to check; see, for example, 
Section~VIII.1, and especially Theorem~VIII.1.3, of \cite{CG}.
(Indeed, the $c$-interval $(-3/4,1/4)$ is precisely the
intersection of the main cardioid of the Mandelbrot set
with the real line.)  Thus, the only rational preperiodic orbits
of $\phi$ must end in fixed points.  However, in part~6
of Theorem~3 of \cite{Poo}, Poonen proved that for $x,c\in\QQ$,
if $x$ is preperiodic to a fixed point, then
$\#\calO(x)\leq 3$, as desired.
%
\end{proof}

Lemmas~\ref{lem:quadpoly} and~\ref{lem:quadpoly2}
are only about the case that the point $x\in\QQ$
is preperiodic for $\phi_c$.
The following result, using similar ideas as those in
Lemma~\ref{lem:quadpoly}, is also relevant to finding
non-preperiodic points of small canonical height.

\begin{lemma}
\label{lem:quadpoly3}
Let $x,c\in\QQ$, let $\phi_c(z) = z^2+c$,
and let $p$ be a prime number.  Set $s:=v_p(c)$,
and suppose $v_p(\phi^i(x)) < \min\{0,s/2\}$ for some $i\geq 0$.
Then
$$\hat{h}_c(x) \geq
\begin{cases}
2^{-i-1}(2-s)  \log p & \text{ if } s\leq -2 \text{ is even},
\\
2^{-1}\log 2 & \text{ if } p=2 \text{ and } s=-2,
\\
\log 2 & \text{ if } p=2 \text{ and } s=-4,
\\
2^{-1}\log p & \text{ otherwise}.
\end{cases}
$$
\end{lemma}


\begin{proof}
Let $r=v_p(x)$.
If $s\geq 0$, suppose first that $r\geq 0$ as well.
Then $v_p(\phi^j(x))\geq 0$ for all $j\geq 0$,
contradicting our hypotheses.  Thus, we must have
$r\leq -1$.  Applying $\phi$, it is immediate that
$v_p(\phi^j(x))=2^j r$ for all $j\geq 1$.
That is, the denominator of $\phi^j(x)$ features $p$ raised
to the power $2^j |r|$, and hence
$$\hat{h}_c(x) = \lim_{j\to\infty} 2^{-j} h(\phi^j(x)) \geq |r| \log p
\geq \log p.$$

If $s\leq -1$ is odd, then regardless
of the value of $r=v_p(x)$, we must have $v_p(\phi(x))\leq s$,
and hence $v_p(\phi^j(x))\leq 2^{j-1} s$ for all $j\geq 1$.  Thus,
$$\hat{h}_c(x) = \lim_{j\to\infty} 2^{-j} h(\phi^j(x)) \geq 2^{-1}\log p.$$

If $s\leq -2$ is even, then the hypothesis that $v_p(\phi^i(x))<s/2$
gives $v_p(\phi^i(x))\leq (s-2)/2$, and hence
$v_p(\phi^j(x))\leq 2^{j-i-1} (s-2)$ for all $j\geq i$.  Thus,
$$\hat{h}_c(x) = \lim_{j\to\infty} 2^{-j} h(\phi^j(x))
\geq 2^{-i-1} (2-s)\log p.$$

If $p=2$ and $s=-2$, write $c=a/4$, where $a\in\QQ$ with $v_2(a)=0$.
If $r\leq -2$, then $v_2(\phi^j(x))=2^j r$
for all $j\geq 0$, and hence $\hat{h}_c(x) \geq -r\log 2 \geq 2\log 2$.
If $r\geq 0$, then $v_2(\phi(x))=-2$, so that 
$\hat{h}_c(\phi(x)) \geq 2\log 2$,
and hence $\hat{h}_c(x) \geq \log 2$.
Lastly, if $r=-1$, suppose first that $a\equiv 1 \pmod{4}$.
Writing $x=b/2$ with $v_2(b)=0$, we have
$v_2(b^2+a)=1$, and therefore $v_2(\phi(x))=-1$ as well.  Thus,
$v_2(\phi^j(x))=-1$ for all $j\geq 0$, contradicting the hypotheses.
Hence, we must instead have $a\equiv 3 \pmod{4}$.
Then $v_2(b^2+1)\geq 2$, so that $v_2(\phi(x))\geq 0$,
implying that $\hat{h}_c(\phi(x))\geq \log 2$,
and therefore that $\hat{h}_c(x)\geq 2^{-1}\log 2$.

It remains to consider $p=2$ and $s=-4$.
If $r\leq -3$, then $v_2(\phi^j(x))=2^j r$
for all $j\geq 0$, and hence $\hat{h}_c(x) \geq -r\log 2 \geq 3\log 2$.
If $r\geq -1$, then $v_2(\phi(x))=-4$, so that
$\hat{h}_c(\phi(x)) \geq 4\log 2$,
and hence $\hat{h}_c(x) \geq 2\log 2$.

Last, if $p=2$, $s=-4$, and $r=-2$,
write $c=a/16$ and $x=b/4$ with $v_2(a)=v_2(b)=0$.
We consider three cases.
First,
if $a\equiv 3 \pmod{8}$, then $v_2(b^2+a)=2$, and hence
$v_2(\phi(x))=-2$; thus, $v_2(\phi^j(x))=-2$ for all $j\geq 0$,
contradicting the hypotheses.
Second,
if $a\equiv 7 \pmod{8}$, then $v_2(b^2+a)\geq 3$, and hence
$v_2(\phi(x))\geq -1$.
By the previous paragraph, then, $\hat{h}_c(\phi(x))\geq 2\log 2$,
and therefore $\hat{h}_c(x)\geq \log 2$.
Third,
if $a\equiv 1 \pmod{4}$, then $v_2(b^2+a)=1$, and hence
$v_2(\phi(x))=-3$.
Again by the previous paragraph, $\hat{h}_c(\phi(x))\geq 3\log 2$,
and therefore $\hat{h}_c(x)\geq (3/2)\log 2$.
\end{proof}

In Algorithm~\ref{alg:quadpoly} below, we will check whether
the denominator of $\phi^3(x)$ is too large --- that is, whether
it is more than the square root of the denominator of $c$.
If that happens, then some prime $p$ appears in the denominator
of $\phi^3(x)$ to too large a power.
By Lemma~\ref{lem:quadpoly3} with $i=3$, then, we would
have
$$\hat{h}_c(x)\geq \min\{2^{-4}\cdot 4, 2^{-1} \}\log p
= \frac{\log p}{4} \geq \frac{\log 3}{4} =0.275\ldots$$
if $p\geq 3$, or
$$\hat{h}_c(x)\geq \min\{2^{-4}\cdot 8, 2^{-1} \}\log 2
= \frac{\log 2}{2} =0.347\ldots.$$
if $p=2$.
Meanwhile, the parameters $c$ in the search range we used had
denominator at most $60060^2$ and absolute value $|c|\leq 10$,
so that $h(c)\leq 2\log(60060)+\log(10)$.
Thus, the associated height ratio would be at least
$$\frac{\hat{h}_c(x)}{h(c)} \geq \frac{\log 3}{4(2\log(60060) + \log 10)}
= 0.0113\ldots$$
and usually much larger.
In short, if the third iterate $\phi^3(x)$ has
the wrong denominator, the pair $(x,c)$ will not be of interest to us.

Similarly, even though Lemmas~\ref{lem:quadpoly} and~\ref{lem:quadpoly2}
do not expressly disallow points
of small but positive canonical height when $c\geq -3/4$, they strongly 
suggest that such points would be unlikely to show
up with denominators in our range of $n\leq 60060$.
We confirmed the absence of such points
for smaller denominators in some preliminary searches.
Thus, although we risked missing a very small handful of points
of small height and larger denominators,
we decided to exclude the parameters $c\geq -3/4$ from our search.

Our entire search, then, was over pairs $(x,c)$
with $x=m/n$ and $c=k/n^2$
fitting the restrictions of Lemmas~\ref{lem:quadpoly}
and~\ref{lem:quadpoly2}.  Meanwhile, since $\phi_c(-x)=\phi_c(x)$,
we may assume that the starting point $x$ is positive.
Thus, we were led to the following search algorithm,
which is guaranteed to find all previously
unknown rational preperiodic pairs
$(x,c)$ in its search range, as well as
nearly all pairs in the same range with especially small canonical
height ratio $\hat{h}_c(x)/h(c)$.

\begin{alg}
\label{alg:quadpoly}
Fix integers $n_{\max}$ and $N_{\max}$ as bounds.
[We used $n_{\max}=60060$ and $N_{\max}=10$.]

1. Let $n$ run through the integers from $1$ to $n_{\max}$.

2. For each such $n$, let $m$ run through the integers from $1$ to $n$
that are relatively prime to $n$.

3. For each such $n$ and $m$, let $k$ run through all the integers
congruent to $m^2$ modulo $n$ and lying between
$-3n^2/4$ and $-n^2 N_{\max}$.
Let $c:=k/n^2$, and $\phi_c(z):=z^2 + c$.

4. Let $x$ run through all rational numbers of the form
$x:=(m+nj)/n$ with $j\in\ZZ$ and $x$ in the interval
$[0,2]$ or $[\sqrt{-c-B},B]$ from
Lemma~\ref{lem:quadpoly}.(d).

5. Compute $\phi^i(x)$ for $i=1,2,3,4$.
Stop iterating and
discard $x$ if $\phi^i(x)$ coincides with an earlier iterate
for some $i\leq 4$,
or has denominator larger than $n$ for some $i\leq 3$.

6. Otherwise, compute
$$ 2^{-12} h\big(\phi^{12}(x)\big).$$
If this value is less than $0.02$ times $h(c)$,
record it as our approximation for the canonical height
$\hat{h}_c(x)$, and record
$\hat{h}_c(x)/h(c)$ as the associated height ratio.
\end{alg}

In \cite{Poo}, Poonen fully classified all quadratic polynomial
preperiodic orbits in $\QQ$ with $\#\calO_c(x)\leq 4$,
which is why we discarded them in Step~5 of the algorithm.
He also showed that for $x,c\in\QQ$, if $x$ is preperiodic for $\phi_c$,
then there are only two ways that
the forward orbit $\calO_c(x)$ can have more than four points.
Specifically, either $c=-29/16$ and $x=\pm 3/4$,
or else the periodic cycle of the orbit has period at least $6$;
the latter should be impossible, according to Conjecture~\ref{conj:poo}.
If such a point did exist, Algorithm~\ref{alg:quadpoly}
would simply declare it to be a point of extremely small
canonical height.  We would have checked
any such points by hand for preperiodicity, but our
search found none.


Our choice of $n_{\max}=60060=2^2\cdot 3\cdot 5\cdot 7 \cdot 11 \cdot 13$
is motivated first in order to ensure the computation finished in reasonable
time, but also because the results of \cite{MS1} show that preperiodic
orbits in $\PP^1(\QQ)$
cannot be very long if there are small primes of good reduction.
Moreover, the results of \cite{Ben9} show that there cannot be too many
rational preperiodic points unless there are a lot of bad primes.
The bad primes of $z^2+c$ are precisely those dividing the denominator
of $c$, and hence we wanted our search to run at least to $n=60060$.

In the case of $p=2$, we needed the $2$ in the denominator of $x$ to
appear to at least the power $2$, for the following reason.
Although any map $\phi_{q/4}(z)=z^2 + q/4$
with $q\in\QQ$ and $q\equiv 1\pmod{4}$ has bad reduction as written,
the conjugate $\phi_{q/4}(z+1/2)-1/2 = z^2 + z + (q-1)/4$
would have good reduction at $2$; by the
results of \cite{Zieve} (summarized in \cite{Sil}, Theorem~2.28),
the forward orbit of a preperiodic point $x\in\QQ$ could then
have length at most $4$.
(On the other hand, if $q\equiv 3\pmod{4}$, then all points
of $\QQ$ would have canonical height at least $\log 2$,
as shown in the proof of Lemma~\ref{lem:quadpoly3}.)

\begin{table}
\scalebox{0.8}{
\begin{tabular}{|c|c|c|c|c|}
\hline
$c$ & $x$ & $\hat{h}_c(x)$ & $\hat{h}_c(x)/h(c)$
&
$\phi_c(x), \phi_c^2(x),\ldots$
\\
\hline
$-\frac{181}{144}$ & $\frac{7}{12}$ & .03433 & .00660 &
$-\frac{11}{12}, -\frac{5}{12}, -\frac{13}{12},
-\frac{1}{12}, -\frac{5}{4}, \frac{11}{36}, -\frac{377}{324},\ldots
\vphantom{\sum_{X_X}^X}$
\\
\hline
$-\frac{1153}{576}$ & $\frac{11}{24}$ & .06505 & .00923 &
$-\frac{43}{24},\frac{29}{24},-\frac{13}{24},-\frac{41}{24},
\frac{11}{12},-\frac{223}{192},-\frac{8021}{12288},\ldots$
\\
\hline
$-\frac{517}{144}$ & $\frac{17}{12}$ & .06885 & .01102 &
$-\frac{19}{12},-\frac{13}{12},-\frac{29}{12},\frac{9}{4},
\frac{53}{36},-\frac{461}{324},-\frac{41093}{26244},\ldots
\vphantom{\sum_{X_X}^X}$
\\
\hline
$-\frac{36989}{19600}$ & $\frac{153}{140}$ & .12319 & .01171 &
$-\frac{97}{140},-\frac{197}{140},\frac{13}{140},-\frac{263}{140},
\frac{1609}{980},\frac{38821}{48020},\ldots
\vphantom{\sum_{X_X}^X}$
\\
\hline
$-\frac{31949}{19600}$ & $\frac{27}{140}$ & .12319 & .01187 &
$-\frac{223}{140},\frac{127}{140},-\frac{113}{140},-\frac{137}{140},
-\frac{659}{980},-\frac{56561}{48020},\ldots
\vphantom{\sum_{X_X}^X}$
\\
\hline
$-\frac{5149}{3600}$ & $\frac{23}{60}$ & .10274 & .01202 & 
$-\frac{77}{60},\frac{13}{60},-\frac{83}{60},\frac{29}{60},
-\frac{359}{300},\frac{13}{7500},-\frac{6704413}{4687500},\ldots
\vphantom{\sum_{X_X}^X}$
\\
\hline
$-\frac{205}{144}$ & $\frac{1}{12}$ & .06866 & .01290 & 
$-\frac{17}{12},\frac{7}{12},-\frac{13}{12},-\frac{1}{4},
-\frac{49}{36},\frac{139}{324},-\frac{32531}{26244},\ldots
\vphantom{\sum_{X_X}^X}$
\\
\hline
$-\frac{181}{144}$ & $\frac{11}{12}$ & .06866 & .01321 & 
$-\frac{5}{12}, -\frac{13}{12},-\frac{1}{12}, -\frac{5}{4},
\frac{11}{36}, -\frac{377}{324},-\frac{2545}{26244}\ldots
\vphantom{\sum_{X_X}^X}$
\\
\hline
$-\frac{16381}{7056}$ & $\frac{97}{84}$ & .13059 & .01346 & 
$-\frac{83}{84},-\frac{113}{84},-\frac{43}{84},-\frac{173}{84},
\frac{1129}{588},\frac{39331}{28812}\ldots
\vphantom{\sum_{X_X}^X}$
\\
\hline
$-\frac{931161001}{476985600}$ & $\frac{30379}{21840}$ & .28548 & .01382 & 
$-\frac{379}{21840},-\frac{42629}{21840},\frac{40571}{21840},
\frac{32731}{21840},\frac{27809}{94640},-\frac{76737829}{41127840},
\ldots\vphantom{\sum_{X_X}^X}$
\\
\hline
$-\frac{10381}{3600}$ & $\frac{121}{60}$ & .12758 & .01380 & 
$\frac{71}{60},-\frac{89}{60},-\frac{41}{60},-\frac{29}{12},
\frac{887}{300},\frac{43947}{7500},\frac{147354737}{4687500},\ldots
\vphantom{\sum_{X_X}^X}$
\\
\hline
$-\frac{9901}{3600}$ & $\frac{131}{60}$ & .12912 & .01403 & 
$\frac{121}{60},\frac{79}{60},-\frac{61}{60},-\frac{103}{60},
\frac{59}{300},-\frac{6779}{2500},\frac{64722611}{14062500},
\ldots
\vphantom{\sum_{X_X}^X}$
\\
\hline
$-\frac{293749}{176400}$ & $\frac{433}{420}$ & .17685 & .01405 & 
$-\frac{253}{420},-\frac{547}{420},\frac{13}{420},-\frac{233}{140},
\frac{6959}{6300},-\frac{630923}{1417500},\ldots
\vphantom{\sum_{X_X}^X}$
\\
\hline
$-\frac{271909}{176400}$ & $\frac{43}{420}$ & .17685 & .01413 & 
$-\frac{643}{420},\frac{337}{420},-\frac{377}{420},-\frac{103}{140},
-\frac{6301}{6300},-\frac{767033}{1417500},\ldots
\vphantom{\sum_{X_X}^X}$
\\
\hline
$-\frac{1513}{576}$ & $\frac{31}{24}$ & .10590 & .01446 & 
$-\frac{23}{24},-\frac{41}{24},\frac{7}{24},-\frac{61}{24},
\frac{23}{6},\frac{2317}{192},\frac{1757219}{12288},\ldots
\vphantom{\sum_{X_X}^X}$
\\
\hline
$-\frac{373}{144}$ & $\frac{23}{12}$ & .08640 & .01459 & 
$\frac{13}{12}, -\frac{17}{12}, -\frac{7}{12},
-\frac{9}{4}, \frac{89}{36},\frac{1141}{324},
\frac{257491}{26244}, \ldots\vphantom{\sum_{X_X}^X}$
\\
\hline
$-\frac{1013082841}{476985600}$ & $\frac{10541}{21840}$ & .30762 & .01483 & 
$-\frac{41299}{21840},\frac{31709}{21840},-\frac{349}{21840},
-\frac{46381}{21840},\frac{677449}{283920},\frac{5994294739}{1679386800}
\ldots\vphantom{\sum_{X_X}^X}$
\\
\hline
$-\frac{160021}{63054}$ & $\frac{181}{252}$ & .17952 & .01498 & 
$-\frac{505}{252},\frac{377}{252},-\frac{71}{252},
-\frac{205}{84},\frac{7793}{2268},\frac{1706041}{183708}
\ldots\vphantom{\sum_{X_X}^X}$
\\
\hline
\end{tabular}
}
\caption{$(x,c)\in\QQ^2$ in the search region
with canonical height ratio $\hat{h}_c(x)/h(c) < 0.015$}
\label{tab:quadpoly}
\end{table}

The results of Algorithm~\ref{alg:quadpoly}
are summarized in Table~\ref{tab:quadpoly},
which lists all pairs $(x,c)$ we found with
height ratio at most $0.015$.
For the sake of accuracy,
we manually computed the canonical heights and height ratios
in the table to many more than the 12 iterates originally
computed by the algorithm.
Note that although there continue to be
pairs with canonical height ratio about $.013$ for
$c$ up to fairly large height (e.g., two on the list with
$n=21840 = 2^4\cdot 3 \cdot 5 \cdot 7 \cdot 13$),
none comes close to the ratio of less than $.007$ attained
by the already-known pair $(7/12,-181/144)$.
Given the large size of this search
(far larger than any prior search for such pairs),
we view this as strong computational support for
Conjecture~\ref{conj:ht}, at least for quadratic
polynomials over $\QQ$.
In addition, as noted above, we found no preperiodic orbits
outside of those classified in \cite{Poo}, thus
providing further computational evidence for
Conjecture~\ref{conj:poo}, and hence for
Conjecture~\ref{conj:ubc} in the case of quadratic polynomials
over $\QQ$.

Our data suggests that just as a small prime of good reduction
limits the length of a $\QQ$-rational preperiodic orbit, such a prime
also appears to
make it difficult for the canonical height $\hat{h}_c(x)$ to be
particularly small.  Indeed, even though our search included pairs $(x,c)$
for which the denominator $n$ of $x$
was not divisible by $4$, the reader may have noticed that
almost all of the pairs in Table~\ref{tab:quadpoly} have
$n$ divisible by $12$, and usually by $60$.
Motivated by this observation,
we did another search allowing height ratios up to $0.03$ when $4\nmid n$.
Table~\ref{tab:quadpolyodd} lists all pairs $(x,c)$ from that search
with height ratio $\hat{h}_c(x)/h(c)$ less than $.025$.
There were only four, only one of which even came close to
the cutoff of $.015$ used for Table~\ref{tab:quadpoly}.
In addition, all four had $n$ divisible by $15$.
The full data from both searches
may be found at
\texttt{http://www3.amherst.edu/\textasciitilde rlbenedetto/quadpolydata/}.

\begin{table}
\scalebox{0.8}{
\begin{tabular}{|c|c|c|c|c|}
\hline
$c$ & $x$ & $\hat{h}_c(x)$ & $\hat{h}_c(x)/h(c)$
&
$\phi_c(x), \phi_c^2(x),\ldots$
\\
\hline
$-\frac{9142351}{5832225}$ & $\frac{1394}{2415}$ & .27774 & .01733 &
$-\frac{2981}{2415},-\frac{106}{2415},-\frac{3781}{2415}
\frac{2134}{2415},-\frac{43699}{55545},-\frac{139366718}{146916525},
\ldots \vphantom{\sum_{X_X}^X}$
\\
\hline
$-\frac{28236091}{12744900}$ & $\frac{5981}{3570}$ & .35856 & .02090 &
$\frac{2111}{3570},-\frac{6661}{3570},\frac{4519}{3570},
-\frac{2189}{3570},-\frac{260493}{141610},\frac{21085814279}{18048052890},
\ldots \vphantom{\sum_{X_X}^X}$
\\
\hline
$-\frac{322457899}{143280900}$ & $\frac{24347}{11970}$ & .48733 & .02487 &
$\frac{22583}{11970},\frac{15667}{11970},-\frac{919}{1710},
-\frac{164371}{83790},\frac{124637651}{78008490},
\ldots \vphantom{\sum_{X_X}^X}$
\\
\hline
$-\frac{136643866}{79655625}$ & $\frac{3854}{8925}$ & .46675 & .02492 &
$-\frac{13646}{8925},\frac{5554}{8925},-\frac{11854}{8925},
\frac{62}{1275},-\frac{9097034}{5310375},
\frac{3820075881834}{3133342515625},
\ldots \vphantom{\sum_{X_X}^X}$
\\
\hline
\end{tabular}
}
\caption{$(x,c)\in\QQ^2$ in the search region
with $v_2(x)\geq -1$
and canonical height ratio $\hat{h}_c(x)/h(c) < 0.025$}
\label{tab:quadpolyodd}
\end{table}

\section{Quadratic Rational Functions}
\label{sect:qrats}

The space $\calM_2$ of conjugacy classes of
degree-$2$ self-morphisms of $\PP^1$ is known, by
work of Milnor \cite{Mil2} and Silverman \cite{Sil2},
to be isomorphic to $\AAA^2$.
In particular, they proved the following result,
which appears as Theorem~4.56 in \cite{Sil}.

\begin{lemma}
\label{lem:phiht}
Let $\phi\in\QQ(z)$ be a rational function of degree $2$.
Then $\phi$ has three fixed points in $\PP^1(\Qbar)$, counting
multiplicity, and the multipliers of the fixed points are precisely
the roots of a cubic polynomial
$$T^3 + \sigma_1(\phi) T^2 + \sigma_2(\phi) T + (\sigma_1(\phi) - 2),$$
where $\sigma_1(\phi),\sigma_2(\phi)\in\QQ$ are certain explicit rational
functions of the coefficients of $\phi$.

Moreover, the function $\calM_2\to\AAA^2$ given by
$\phi\mapsto (\sigma_1(\phi),\sigma_2(\phi))$ is an isomorphism
of algebraic varieties defined over~$\QQ$.
\end{lemma}

The precise formulas for $\sigma_1$ and $\sigma_2$ in terms
of the coefficients of $\phi$ are not important here, but
the interested reader can find them on page~189 of \cite{Sil}.
Instead, what \emph{is} important is that Lemma~\ref{lem:phiht}
allows us to define a height function on $\calM_2$, given by
\begin{equation}
\label{eq:phiht}
h(\phi) = h(\sigma_1(\phi),\sigma_2(\phi)) :=
\max\{|a|,|b|,|c|\},
\end{equation}
where $\sigma_1(\phi)=a/c$ and $\sigma_2(\phi)=b/c$,
with $a,b,c\in\ZZ$ and $\gcd(a,b,c)=1$.

However, we are interested instead in the moduli space
not of morphisms $\phi$, but rather of pairs $(x,\phi)$
consisting of a point $x$ and a morphism $\phi$.
More precisely, given a field $K$, we define
$$\calP\calM_2(K) := \{(x,\phi)\in\PP^1(K)\times K(z) : \deg\phi=2\}/\sim,$$
where the equivalence relation $\sim$ is given by
$$(x,\phi) \sim \big(\eta(x),\eta\circ\phi\circ \eta^{-1}\big)
\quad
\text{for any } \eta\in\PGL(2,K).$$
Of course, to verify that this quotient is well-behaved from the
perspective of algebraic geometry, the machinery of
geometric invariant theory is required.  However, given the
parametrization we are about to construct, and given that we
are really interested in $\calP\calM_2$ merely as a set, we can
safely sidestep these issues.

Our parametrization, based on an idea suggested by Elkies \cite{Elk},
begins with an informal dimension count.
The set $\calP\calM_2$ should be a space of dimension $3$, because
there are five dimensions of choices for $\phi$ (the six coefficients
of the rational function, minus one for multiplying both top and bottom
by an element of $K$) and one for $x$, but then the quotient by
the action of $\PGL(2,K)$ subtracts three dimensions.

Since we will be interested in both infinite orbits of small canonical
height, and in preperiodic points with long forward orbits, we will
restrict our attention to pairs $(x,\phi)$ for which the the first
six iterates $\{\phi^i(x) : 0\leq i\leq 5\}$ are all distinct.
(Incidentally, the moduli spaces corresponding to pairs $(x,\phi)$
where $x$ is periodic of period $5$ or less are all birational
over $\QQ$ to $\PP^2$, as shown in Theorem~1(1) of \cite{BC}.)
Call the subset of $\calP\calM_2$ consisting of (equivalence classes of)
such pairs $\calP\calM'_2$.  Given $(x,\phi)\in\calP\calM'_2$,
and writing $x_i:=\phi^i(x)$, then, there is a unique $\eta\in\PGL(2,K)$
such that $\eta(x_0)=\infty$, $\eta(x_1)=1$, and $\eta(x_2)=0$.
Hence, the equivalence class of $(x,\phi)$ under $\sim$ contains a unique
element whose orbit begins in the form
\begin{equation}
\label{eq:orbit6}
\infty \mapsto 1 \mapsto 0 \mapsto x_3 \mapsto x_4 \mapsto x_5.
\end{equation}
Thus, we can view $\calP\calM'_2$ 
as coinciding, as a set, with the Zariski open subset of $\AAA^3(K)$
consisting of triples $(x_3,x_4,x_5)$
for which the map $\phi$ giving 
the partial orbit~\eqref{eq:orbit6}
actually has degree $2$.
(In particular, the three coordinates must be distinct,
with none equal to $0$ or $1$, and $\phi$ must not degenerate
to a lower-degree map.)

The restrictions $\phi(\infty)=1$ and $\phi(1)=0$ dictate that the
associated map $\phi$ is of the form
\begin{equation}
\label{eq:phiform}
\phi(z) = \frac{(a_1 z + a_0)(z-1)}{a_1 z^2 + b_1 z + b_0}.
\end{equation}
Moreover, it is straightforward to check
that \eqref{eq:orbit6} stipulates
\begin{align}
\label{eq:coeffmess}
a_1 & =x_4 (x_3^2 x_4 - x_3^2 x_5 - x_3^2 + 2 x_3 x_5 - x_4 x_5),
\notag \\
a_0 &= -x_3 b_0 = x_3^2 x_4(x_4 - 1)(x_5 -x_4 + x_4 x_5 - x_3 x_5),
\\
b_1 &= x_3^2 x_4^2 x_5 - x_3^3 x_4^2
+ 2x_3^3 x_4 - x_3^2 x_4^2 + x_4^3 x_5
- x_3^2 x_4 x_5 - x_3 x_4^2 x_5
\notag\\
& \quad - x_3^3 + x_3 x_4^2 - x_4^3 + x_3^2 x_5
+ x_3^2 - x_3 x_4 + x_4^2 - x_3 x_5.
\notag
\end{align}
Of course,
the Zariski closed subset of $\AAA^3$ on which $\phi$
does not have degree $2$ is precisely the zero locus of the resultant
of the two polynomials $(a_1 z + a_0)(z-1)$ and $a_1 z^2 + b_1 z + b_0$.

\begin{defin}
\label{def:ppspace}
For any integers $m\geq 0$ and $n\geq 1$ with $m+n\geq 6$,
let $X_{m,n}$ denote the
space of triples $(x_3,x_4,x_5)\in\calP\calM'_2\subseteq\AAA^3$
for which the map $\phi$ given
by equations~\eqref{eq:phiform} and~\eqref{eq:coeffmess}
satisfies
$\phi^{m+n}(\infty)=\phi^m(\infty)$, 
but $\phi^i(\infty)\neq\phi^j(\infty)$ for any other $i\neq j$
between $0$ and $m+n$.
\end{defin}

As a set, then, $X_{m,n}$ is the moduli space of pairs $(x,\phi)$
up to coordinate change for which $\phi^m(x)$ is periodic
of minimal period $n$, but $\phi^i(x)$ is not periodic for
$0\leq i\leq m-1$.  (Presumably it is in fact birational to the
actual geometric moduli space of such pairs up to coordinate change,
but as we noted before, we will not need to know that here.)

Our preliminary searches showed an abundance of orbits
with $\phi^6(\infty)=\phi^4(\infty)$.  This observation led us
to the following result.

\begin{lemma}
\label{lem:4simp}
$X_{4,2}$ is birational to $\PP^2$.
\end{lemma}

\begin{proof}
$X_{4,2}$ is the locus of triples for which $\phi(x_5)=x_4$.
Solving this equation using the formula for $\phi$ from
equations~\eqref{eq:phiform} and~\eqref{eq:coeffmess} gives
$$x_4(x_4 -1)(x_3-x_5)(x_4-x_5)(x_4-x_4 x_5 - x_3^2 x_5 + 2 x_3 x_5 - x_3)
=0.$$
Dividing out by the first four factors, which correspond to parameter
choices outside $\calP\calM'_2$, we get
\begin{equation}
\label{eq:x5for42}
x_5 = \frac{x_4 - x_3}{x_3^2 - 2x_3 + x_4}.
\end{equation}
Thus, $X_{4,2}$ is parametrized by $(x_3,x_4)$, with $x_5$
given by equation~\eqref{eq:x5for42}.
\end{proof}

\begin{remark}
With $x_5$ given by equation~\eqref{eq:x5for42}, it is easy
to check with computational software that 
the resultant of the numerator and denominator of
the corresponding map $\phi$ from
equations~\eqref{eq:phiform} and~\eqref{eq:coeffmess} is
\begin{equation}
\label{eq:42res}
x_3^2 x_4^2 (x_3-x_4) (x_3-1)^8(x_4-1)^2 (x_3^2 -2x_3 + x_4)
(x_3 x_4-x_3+x_4)^5
\end{equation}
It is fairly clear why most of the terms in
expression~\eqref{eq:42res} appear, when we recall that none of
$x_3,x_4,x_5$ can equal $0$, $1$, or $\infty$.
The term $x_3 x_4-x_3+x_4$, meanwhile, appears because it is
zero precisely when the degree $1$ map taking
$\infty\mapsto 1\mapsto 0\mapsto x_3$ already maps $x_3$ to $x_4$,
and hence $\phi$ degenerates.
\end{remark}

Bearing in mind everything so far in this section,
we are led to the following algorithm.

\begin{alg}
\label{alg:quadrat}
Fix a height bound $B>0$ and a threshold $r>0$ for small height ratios.
[We used $B=\log(100)$ and $r=.002$.]

1. Let $(x_3,x_4,x_5)$ run through all triples in
$\QQ^3$ with $h(x_3),h(x_4),h(x_5)\leq B$
and with $0,1,x_3,x_4,x_5$ all distinct.

2. For each triple, discard it if the resulting map $\phi$
of equations~\eqref{eq:phiform} and~\eqref{eq:coeffmess}
degenerates to degree less than $2$.
In light of Lemma~\ref{lem:4simp},
also discard the triple if $\phi(x_5)=x_4$.

3. Compute $\phi^i(\infty)=\phi^{i-5}(x_5)$
for $i=6,\ldots,10$.  For each such $i$,
stop iterating and record $(\infty,\phi)$ as preperiodic if
$\phi^i(\infty)$ coincides with an earlier iterate.

4. Let $C=\log(|R|/D)>0$
be the constant of Lemma~\ref{lem:htfunc} for $\phi$,
and let $h(\phi)\geq 0$ be the height of $\phi$ from
Lemma~\ref{lem:phiht} and equation~\eqref{eq:phiht}.

5. Discard the triple $(x_3,x_4,x_5)$ if
$$h\big( \phi^8(\infty) \big) \geq 2^8 r\cdot h(\phi) + C
\quad\text{or}\quad
h\big( \phi^{10}(\infty) \big) \geq 2^{10} r\cdot h(\phi) + C,$$
since either one forces the height ratio
$\hat{h}_{\phi}(\infty)/h(\phi)$
to be larger than the threshold $r$.

6. Otherwise, compute
$$ 2^{-15} h\big(\phi^{15}(\infty)\big).$$
If this value is less than $r\cdot h(\phi) + 2^{-15}C$,
record it as our approximation for the canonical height
$\hat{h}_{\phi}(\infty)$, and record
$\hat{h}_{\phi}(x)/h(\phi)$ as the associated height ratio.
\end{alg}

See
\texttt{http://www3.amherst.edu/\textasciitilde rlbenedetto/quadratdata/}
for the full data we found with Algorithm~\ref{alg:quadrat}.
The most noticeable feature of the data was the
huge number of points on $X_{5,2}$.
We therefore analyzed $X_{5,2}$, leading us to the following result,
which is the precise form of the Theorem stated in the introduction.

\begin{thm}
\label{thm:52ell}
$X_{5,2}$ is birational to the elliptic surface
\begin{equation}
\label{eq:52ell}
E: y^2 = 4x^3 + (4t^4 + 4t^3 + 1) x^2
- 2t^3(t+1)^2 (2t^2 + 2t + 1) x + t^6 (t+1)^4.
\end{equation}
Moreover, $E$ has positive rank over $\QQ(t)$, including the
non-torsion point $P$ given by $(x,y)= (0,t^3(t+1)^2)$.
\end{thm}

\begin{proof}
There is a morphism $X_{5,2}\to X_{4,2}$ taking a pair
$(x,\phi)$ to the pair $(\phi(x),\phi)$.  Parametrizing
$X_{4,2}$ by $(x_3,x_4)$ via equation~\eqref{eq:x5for42}
as in the proof of Lemma~\ref{lem:4simp}, and letting $\psi$
denote the morphism giving the orbit
$\infty,1,0,x_3,x_4,x_5,x_4$,
we can parametrize $X_{5,2}$ by triples $(x_3,x_4,w)$ for which
$\psi(w)=\infty$.
That is, $X_{5,2}$ is the subvariety of
$\AAA^3$ defined by $g(x_3,x_4,w)=0$, where
\begin{multline*}
g(x_3,x_4,w)
= w(1-x_4)x_3^3
+ (1-w) x_3 x_4^2
\\
+ w(x_4 w + x_4 - 2)x_3^2
- (2w^2 - 2w + 1)x_3 x_4 + w x_3 + (w^2-w)x_4^2
\end{multline*}
is the denominator of $\psi(w)$.

This variety is singular along the line $x_3=w=1$.  Blowing up along this
line via $(x_3-1)s = w-1$ gives the surface
$$s(s-1)x_3^2 x_4 + s x_3^2
-(2s^2-2s+1)x_3 x_4 - (s-1)x_3 + s(s-1)x_4^2=0,$$
which is cubic in $(x_3,x_4)$.
Standard manipulations, along with the substitution $s=t+1$,
give equation~\eqref{eq:52ell}, with the birationality given by
\begin{align*}
t &= \frac{w-x_3}{x_3-1},
\qquad
x = \frac{(w-x_3)^2(w-1)}{x_3(x_3-1)^3},
\\
y &= \frac{(w-x_3)^2(w-1)}{x_3^2(x_3-1)^5}\big[
(w-x_3)(w-1)(2x_4 + x_3^2) - x_3\big( (w-x_3)^2 + (w-1)^2 \big) \big]
\end{align*}
and
\begin{align*}
w &= \frac{t}{x}\big(t(t+1)^2 - x\big),
\qquad
x_3 = \frac{t^2(t+1)}{x},
\\
x_4 &= \frac{t}{2x^2}\big(y + (2t^2 + 2t+1)x - t^3(t+1)^2\big).
\end{align*}

%
%
%
%

The point $P$ in the statement of the
Theorem obviously lies on the $\QQ(t)$-curve $E$ of equation~\eqref{eq:52ell}.
Finally, specializing $E$ at $t=1$ gives an isomorphic copy
of curve 142a1 in Cremona's tables, which has trivial torsion
and rank~1 over $\QQ$.  The point $P$ specializes to
the generator, and hence $P$ must have infinite order in $E(\QQ(t))$.
\end{proof}

Obviously the point
$P$ itself lies in the degeneracy locus of the birationality
of Theorem~\ref{thm:52ell}, given the formula $x_3=t^2(t+1)/x$.
Meanwhile, it is easy to compute the first few multiples of $P$ on $E$:
\begin{align*}
P &= \big(0,t^3(t+1)^2\big),
&
[2]P &= \big(t(t+1)^2,t(t+1)^4\big),
\\
[3]P &= \big(-t^2(t+1)^2,t^2(t+1)^3\big),
&
[4]P &= \big(t^2(t+1), -t^2(t+1)(t^2+t-1)\big),
\end{align*}
\begin{align*}
[5]P &= \big(-t(t+1),-t(t+1)(t^3+3t^2+2t-1)\big),
\\
[6]P &= \big(t^3(t+1)^2(t+2),-t^3(t+1)^3(2t^3 + 6t^2 + 4t - 1)\big),
\end{align*}
and it turns out that all of the points $[n]P$,
for $n\in\{0,\pm 1,\pm 2, \pm 3, \pm 4, \pm 5\}$,
also lie in the degeneracy locus.  However, all but
finitely many $[n]P$ must lie off it, and in particular,
both $[6]P$ and $[-6]P$ are nondegenerate.  $[6]P$,
for example, corresponds to
$$x_3 = \frac{1}{t(t+1)(t+2)},
\quad
x_4 = - \frac{(t^2 + t-1)}{t^2(t+2)^2},
\quad
x_5 = \frac{t+1}{t+2},$$
with pole
$$w= - \frac{(t+1)(t^2+t-1)}{t(t+2)}.$$

Meanwhile,
given the results of \cite{BC} on the infinitude of conjugacy classes
admitting a rational $6$-periodic point, there are also
certainly infinitely many
rational functions with a length $7$ preperiodic orbit.
Indeed, given any non-critical point $x$ in the $6$-cycle (which will
be most, and usually all, of them), the other preimage $y$ of $\phi(x)$ is
necessarily rational and strictly preperiodic with orbit length $7$.
Thus, $(y,\phi)$ gives a point in $X_{1,6}(\QQ)$.

We also found numerous rational points
(though not nearly as many as for $X_{5,2}$)
on the other moduli surfaces corresponding
to strictly preperiodic orbits of length~7:
$X_{6,1}$, $X_{4,3}$, $X_{3,4}$, and $X_{2,5}$
It would be interesting to find descriptions
of these other moduli surfaces, some of which
probably have infinitely many rational points.
By contrast, the only rational points we found on $X_{0,7}$,
the moduli space of maps with $7$-cycles,
corresponded to conjugates of the map $\phi$ of
equation~\eqref{eq:7per}.

As for preperiodic orbits of length~8, certainly $X_{1,7}(\QQ)$
is nonempty.  After all, each of the seven strictly preperiodic
points of the map $\phi$ of equation~\eqref{eq:7per} has forward
orbit of length~8.
The resulting elements of $X_{1,7}$ did not show up directly
in our search, however, because the coordinate changes required
(to move any such preperiodic point $x$ to $\infty$, $\phi(x)$ to $1$,
and $\phi^2(x)$ to $0$) result in triples $(x_3,x_4,x_5)$ outside
our search region.

We found several other triples $(x_3,x_4,x_5)$ corresponding to
rational forward orbits of length $8$; they are listed in
Table~\ref{tab:quadratlen8}.  We found 26 points
in $X_{6,2}$, two in $X_{5,3}$, and none in any other surface
$X_{m,n}$ with $m+n=8$.  (It should be expected that $2$-cycles
are the easiest to realize, as a degree-two map has three fixed
points, six $3$-periodic points, and many more of any higher
period, but only two $2$-periodic points.  Thus, the $2$-periodic
points are roots of a quadratic polynomial, while all the
other periodic points are roots of higher-degree polynomials.)

Note, however, that for $m\geq 1$,
points in the moduli space $X_{m,n}$ generally come in sets of two.
After all, if $(x,\phi)$ is
a point in the moduli space, then unless $x$ is a critical point,
$\phi(x)$ has two rational preimages: $x$, and some other point $y$.
Thus, $(y,\phi)$ is also a point in the same moduli space.
The two points we found in $X_{5,3}$ were related to each other
in this way, as were 11 sets of two points we found in $X_{6,2}$.
In Table~\ref{tab:quadratlen8}, therefore,
we have listed only one point in $X_{5,3}$,
and fifteen in $X_{6,2}$.  The four unpaired points we found
in $X_{6,2}$ appear last in the table; the points they
would be paired with lie outside our search region.

\begin{table}
\scalebox{0.8}{
\begin{tabular}{|c|c|c|c|c|}
\hline
$\phi$ & orbit & tail & period & total length
\\
\hline
$\frac{330 z^2 - 187z - 143}{330 z^2 + 1217z + 429} \vphantom{\sum_{X_X}^X}$
&
$\infty,1,0,
-\frac{1}{3},
-\frac{11}{15},
-\frac{3}{5},
-\frac{55}{114},
-\frac{13}{44},
-\frac{3}{5}$
& 5 & 3 & 8
\\
\hline
$\frac{21z^2 - 84z + 63}{21z^2 - 16z - 21}
\vphantom{\sum_{X_X}^X}$
&
$\infty,1,0,
-3,
\frac{7}{3},
-\frac{1}{3},
-7,
-\frac{3}{2},
-7$
& 6 & 2 & 8
\\
\hline
$\frac{52z^2 - 30z - 22}{52z^2 + 245z + 88}
\vphantom{\sum_{X_X}^X}$
&
$\infty,1,0,
-\frac{1}{4},
-\frac{3}{8},
-1,
-\frac{4}{7},
-\frac{9}{26},
-\frac{4}{7}$
& 6 & 2 & 8
\\
\hline
$\frac{120z^2 - 98z - 22}{120z^2 + 749z + 132}
\vphantom{\sum_{X_X}^X}$
&
$\infty,1,0,-\frac{1}{6},-\frac{2}{9},-\frac{1}{5},
-\frac{12}{65},-\frac{1}{12},-\frac{12}{65}$
& 6 & 2 & 8
\\
\hline
$\frac{30z^2 - 10z - 20}{30z^2 + 7z - 30}
\vphantom{\sum_{X_X}^X}$
&
$\infty,1,0,\frac{2}{3},\frac{10}{9},\frac{2}{5},
\frac{6}{7},\frac{10}{3},\frac{6}{7}$
& 6 & 2 & 8
\\
\hline
$\frac{33z^2 - 429z + 396}{33z^2 - 197z + 132}
\vphantom{\sum_{X_X}^X}$
&
$\infty,1,0,
3,\frac{11}{3},5,33,
\frac{3}{4},33$
& 6 & 2 & 8
\\
\hline
$\frac{176z^2 + 1397z - 1573}{176z^2 + 500z - 1144}
\vphantom{\sum_{X_X}^X}$
&
$\infty,1,0,
\frac{11}{8},
-\frac{11}{2},
-\frac{11}{4},
\frac{55}{16},
2,
\frac{55}{16}$
& 6 & 2 & 8
\\
\hline
$\frac{1350z^2 - 837z - 513}{1350z^2 + 5585z + 1710}
\vphantom{\sum_{X_X}^X}$
&
$\infty,1,0,-\frac{3}{10},-\frac{9}{10},-\frac{3}{5},
-\frac{72}{175},-\frac{1}{6},-\frac{72}{175}$
& 6 & 2 & 8
\\
\hline
$\frac{700z^2 - 95z - 605}{700z^2 + 1336z + 880}
\vphantom{\sum_{X_X}^X}$
&
$\infty,1,0,-\frac{11}{16},-\frac{5}{7},
-\frac{7}{11},-\frac{5}{6},-\frac{11}{70},-\frac{5}{6}$
& 6 & 2 & 8
\\
\hline
$\frac{784z^2 - 416z - 368}{784z^2 + 3885z + 644}
\vphantom{\sum_{X_X}^X}$
&
$\infty,1,0,
-\frac{4}{7},
-\frac{2}{21},
-\frac{8}{7},
-\frac{20}{49},
\frac{1}{12},
-\frac{20}{49}$
& 6 & 2 & 8
\\
\hline
$\frac{1428z^2 - 1668z + 240}{1428z^2 - 1723z + 900}
\vphantom{\sum_{X_X}^X}$
&
$\infty,1,0,\frac{4}{15},-\frac{4}{21},\frac{10}{21},
-\frac{4}{7},\frac{12}{17},-\frac{4}{7}$
& 6 & 2 & 8
\\
\hline
$\frac{308z^2 + 19292z - 19600}{308z^2 + 1937z + 7700}
\vphantom{\sum_{X_X}^X}$
&
$\infty,1,0,-\frac{28}{11},-14,-\frac{28}{5},
-\frac{308}{17},-\frac{40}{11},-\frac{308}{17}$
& 6 & 2 & 8
\\
\hline
$\frac{9009z^2 - 17094z + 8085}{9009z^2 - 18454z - 10395}
\vphantom{\sum_{X_X}^X}$
&
$\infty,1,0,-\frac{7}{9},\frac{77}{27},\frac{35}{11},
\frac{63}{31},-\frac{77}{78},\frac{63}{31}$
& 6 & 2 & 8
\\
\hline
$\frac{5712z^2 - 5937z + 225}{5712z^2 - 137612z + 5400}
\vphantom{\sum_{X_X}^X}$
&
$\infty,1,0,\frac{1}{24},\frac{1}{26},\frac{3}{68},
\frac{117}{2992},\frac{2}{35},\frac{117}{2992}$
& 6 & 2 & 8
\\
\hline
$\frac{51480z^2 + 910z - 52390}{51480z^2 + 275477z - 120900}
\vphantom{\sum_{X_X}^X}$
&
$\infty,1,0,\frac{13}{30},-\frac{26}{5},-\frac{91}{11},
\frac{780}{253},\frac{13}{36},\frac{780}{253}$
& 6 & 2 & 8
\\
\hline
$\frac{24255z^2 - 277830z + 253575}{24255z^2 + 314788z + 65205}
\vphantom{\sum_{X_X}^X}$
&
$\infty,1,0,\frac{35}{9},-\frac{5}{18},-\frac{49}{3},
\frac{105}{13},-\frac{15}{154},\frac{105}{13}$
& 6 & 2 & 8
\\
\hline
\end{tabular}
}
\caption{$(\infty,\phi)\in\calP\calM'_2(\QQ)$ in the search region
with preperiodic orbit of length (at least) 8}
\label{tab:quadratlen8}
\end{table}

Accounting for both preimages of each point in the strict
forward orbit of $\infty$, 
each of the maps listed in Table~\ref{tab:quadratlen8}
comes with $14$ rational preperiodic points, assuming none
of the points involved are critical images.  And a simple
but tedious computation shows that this assumption is accurate.
(In fact, none of the maps in Table~\ref{tab:quadratlen8}
has any $\QQ$-rational critical points at all.)
Moreover, a longer computation --- using
Hutz's algorithm \cite{Hutz} for computing
$\Preper(\phi,\QQ)$  ---
shows that in each case,
$\Preper(\phi,\QQ)$ is precisely this set of 14 points.
The foregoing evidence from our data led us to propose
parts~(a)--(c) of Conjecture~\ref{conj:qrat}.

The abundance of points in $X_{6,2}$ suggests that
$X_{6,2}$ may be infinite, although we have not yet found a
proof or disproof of this statement.
Meanwhile, the surfaces $X_{m,n}$ for other $m,n$ with $m+n=8$
all appear to have at most finitely many rational points.
However, here the limitations of our search region, even though
it is quite large, should inspire some caution.  In particular,
as noted above, we know that $X_{1,7}$ must have some $\QQ$-rational
points, but they all lie outside our region.  That is, all such points
$(x_3,x_4,x_5)$ have $\max\{h(x_3),h(x_4),h(x_5)\} > \log 100$.


We close with Table~\ref{tab:quadratht}, which gives the pairs
$(x,\phi)$ in our search region with the smallest positive height
ratios $\hat{h}_{\phi}(x)/h(\phi)$ that we found --- specifically,
with ratio less than $.0012$.
(Recall that $h(\phi)$ is the height of $\phi$ itself as given by
Lemma~\ref{lem:phiht} and equation~\eqref{eq:phiht}.
As in Tables~\ref{tab:quadpoly} and~\ref{tab:quadpolyodd},
we manually computed the canonical heights and height ratios
in the table to many more than the 15 iterates originally
computed by the algorithm.)
In addition, as with the preperiodic points
in Table~\ref{tab:quadratlen8}, pairs $(x,\phi)$
with small height ratio come in sets of two;
Table~\ref{tab:quadratht} only lists one of the two pairs
in this situation.

Incidentally,
the fifth map listed in Table~\ref{tab:quadratht}
is simply a coordinate change of the first!
More precisely, calling the first pair $(\infty,\psi)$, the fifth
is the equivalence class of $(\psi(\infty),\psi)$, and
the canonical height ratio of the latter pair is exactly
double that of $(\infty,\psi)$.
This pair $(\infty,\psi)$ provides further
evidence for Conjecture~\ref{conj:ht}.  First discovered
by Elkies in a similar but smaller-scale search
\cite{Elk}, it showed up very
early in our search, because the point
$(\infty,\psi)\in\calP\calM'_2$ is represented
by the small-height triple $(x_3,x_4,x_5)=(-1/3, -1/5, -3/5)$.
It is telling that the second-place candidate
had height ratio one and a half times as large, even though our
search extended to points in $\calP\calM'_2$ of far
larger height.  This observation led us to pose part~(d)
of Conjecture~\ref{conj:qrat}.  After all, if there were pairs
$(x,\phi)\in\calP\calM'_2$ with even smaller positive
canonical height ratio, presumably at least \emph{one}
of them would have shown up.

\begin{table}
\scalebox{0.8}{
\begin{tabular}{|c|c|c|c|}
\hline
$\phi$ & $\hat{h}_{\phi}(\infty)$ & $\hat{h}_{\phi}(\infty)/h(\phi)$
&
$\phi^3(x), \phi^4(x),\ldots$
\\
\hline
$\frac{10z^2 - 7z - 3}{10z^2 + 37z + 9}
\vphantom{\sum_{X_X}^X}$
&
.00360 & .000466
&
$-\frac{1}{3}, -\frac{1}{5}, -\frac{3}{5}, -\frac{1}{2},
-\frac{3}{7}, -\frac{15}{41}, -\frac{9}{32}, -\frac{41}{105},
\ldots \vphantom{\sum_{X_X}^X}$
\\
\hline
$\frac{48165z^2 - 54663z + 6498}{48165z^2 - 49361z + 1482}
\vphantom{\sum_{X_X}^X}$
&
.01425 & .000747
&
$\frac{57}{13}, \frac{38}{39}, \frac{76}{65}, \frac{57}{65},
\frac{12}{13},\frac{2109}{2197}, \frac{12084}{11525},
\ldots \vphantom{\sum_{X_X}^X}$
\\
\hline
$\frac{91z^2 + 399z - 490}{91z^2 - 16z - 350}
\vphantom{\sum_{X_X}^X}$
&
.01221 & .000867
&
$\frac{7}{5}, -\frac{14}{11}, \frac{14}{3}, \frac{28}{13},
21,\frac{28}{23}, -\frac{329}{591}, \frac{975016}{446071},
\ldots \vphantom{\sum_{X_X}^X}$
\\
\hline
$\frac{1701z^2 - 427z - 1274}{1701z^2 - 3222z + 546}
\vphantom{\sum_{X_X}^X}$
&
.01553 & .000919
&
$-\frac{7}{3}, \frac{14}{27}, \frac{14}{9}, -\frac{56}{9},
\frac{7}{9}, \frac{106}{171}, \frac{4424}{3987},
-\frac{6245167}{16848423},
\ldots \vphantom{\sum_{X_X}^X}$
\\
\hline
$\frac{7z^2 - 6z - 1}{7z^2 + 20z - 3}
\vphantom{\sum_{X_X}^X}$
&
.00721 & .000931
&
$\frac{1}{3}, -\frac{1}{2}, -\frac{1}{3}, -\frac{1}{5}, -\frac{1}{14},
\frac{5}{41}, \frac{57}{16}, \frac{3403}{8043},
\ldots \vphantom{\sum_{X_X}^X}$
\\
\hline
$\frac{60z^2 - 24z - 36}{60z^2 - 143z + 6}
\vphantom{\sum_{X_X}^X}$
&
.01128 & .000935
&
$-6, \frac{3}{4}, \frac{3}{10}, \frac{6}{5}, -\frac{3}{11},
-\frac{48}{95}, -\frac{78}{853}, -\frac{96159}{56528},
\ldots \vphantom{\sum_{X_X}^X}$
\\
\hline
$\frac{42z^2 - 67z + 25}{42z^2 - 75z + 30}
\vphantom{\sum_{X_X}^X}$
&
.00829 & .000958
&
$\frac{5}{6}, \frac{1}{2}, \frac{2}{3}, \frac{3}{4},
\frac{13}{21}, \frac{8}{7}, -\frac{23}{6}, \frac{2697}{2804},
\ldots \vphantom{\sum_{X_X}^X}$
\\
\hline
$\frac{845z^2 - 20z - 825}{845z^2 + 3302z + 2145}
\vphantom{\sum_{X_X}^X}$
&
.01243 & .001028
&
$-\frac{5}{13}, -\frac{9}{13}, -\frac{20}{13}, -\frac{285}{221},
-\frac{345}{403}, \frac{510}{169}, \frac{1395}{4057},
-\frac{829730}{3831763},
\ldots \vphantom{\sum_{X_X}^X}$
\\
\hline
$\frac{8190z^2 - 10983z + 2793}{8190z^2 - 23941z + 7182}
\vphantom{\sum_{X_X}^X}$
&
.02187 & .001040
&
$\frac{7}{18}, \frac{7}{26}, \frac{21}{65}, \frac{28}{85},
\frac{9}{26}, \frac{119}{537}, \frac{50022}{149725},
\ldots \vphantom{\sum_{X_X}^X}$
\\
\hline
$\frac{1625z^2 - 1417z - 208}{1625z^2 + 11075z + 260}
\vphantom{\sum_{X_X}^X}$
&
.01780 & .001101
&
$-\frac{4}{5}, -\frac{13}{50}, -\frac{7}{65}, \frac{1}{25},
-\frac{13}{35}, -\frac{71}{475}, -\frac{1183}{40085},
\frac{113139661}{44931025},
\ldots \vphantom{\sum_{X_X}^X}$
\\
\hline
$\frac{14z^2 - 9z - 5}{14z^2 + 76z + 20}
\vphantom{\sum_{X_X}^X}$
&
.01192 & .001125
&
$-\frac{1}{4}, -1, -\frac{3}{7}, -\frac{1}{7}, -\frac{4}{11},
-\frac{3}{140}, -\frac{611}{2339}, -\frac{308865}{201037},
\ldots \vphantom{\sum_{X_X}^X}$
\\
\hline
$\frac{945z^2 - 980z + 35}{945z^2 - 3372z + 315}
\vphantom{\sum_{X_X}^X}$
&
.02104 & .001145
&
$\frac{1}{9}, \frac{35}{27}, -\frac{1}{7}, \frac{5}{21}, \frac{1}{3},
\frac{35}{132}, \frac{26287}{85071}, \frac{86934355}{311687886},
\ldots \vphantom{\sum_{X_X}^X}$
\\
\hline
$\frac{90z^2 - 20826z + 20736}{90z^2 - 533z + 1080}
\vphantom{\sum_{X_X}^X}$
&
.02642 & .001161
&
$\frac{96}{5},-\frac{72}{5}, \frac{99}{8}, -27, 8,
-\frac{1251}{23}, \frac{131004}{27343},
-\frac{532888651446}{4100060723},
\ldots \vphantom{\sum_{X_X}^X}$
\\
\hline
$\frac{1375z^2 - 847z - 528}{1375z^2 + 4025z + 1320}
\vphantom{\sum_{X_X}^X}$
&
.01778 & .001183
&
$-\frac{2}{5}, -\frac{11}{25}, -\frac{3}{5}, -\frac{99}{125}, -1,
-\frac{121}{95}, -\frac{19017}{10775}, -\frac{36280937}{10371365},
\ldots \vphantom{\sum_{X_X}^X}$
\\
\hline
$\frac{1400z^2 - 2005z + 605}{1400z^2 - 3148z + 1100}
\vphantom{\sum_{X_X}^X}$
&
.01621 & .001186
&
$\frac{11}{20}, \frac{5}{14}, \frac{7}{16}, \frac{5}{11}, \frac{23}{56},
\frac{110}{269}, \frac{2809}{6844}, \frac{23315575}{56982922},
\ldots \vphantom{\sum_{X_X}^X}$
\\
\hline
$\frac{290z^2 - 966z + 676}{290z^2 - 1077z + 754}
\vphantom{\sum_{X_X}^X}$
&
.02290 & .001198
&
$\frac{26}{29}, 2, \frac{2}{5}, \frac{10}{11}, \frac{18}{7},
-\frac{122}{109}, \frac{190862}{209065}, \frac{26084280098}{9085163135},
\ldots \vphantom{\sum_{X_X}^X}$
\\
\hline
\end{tabular}
}
\caption{$(\infty,\phi)\in\calP\calM'_2(\QQ)$ in the search region
with positive canonical height ratio less than $.0012$.}
\label{tab:quadratht}
\end{table}

\begin{remark}
\label{rem:arbitrary}
The decision to check the height
of $\phi^i(\infty)$ for iterates $i=8,10,15$
in Algorithm~\ref{alg:quadrat} was somewhat
arbitrary, but the goal was to minimize
the algorithm's run time.  In practice,
most points tested tend to blow up in height rapidly
under iteration, and therefore we did a preliminary
check only a few iterations after $x_5$ and another
two iterations later, to avoid computing
ten iterates of $x_5$ for \emph{every} candidate.
On the other hand, the computation time required
to test \emph{all} the iterates
$i=6,7,\ldots$ for each candidate also slowed
the program down.  Thus, we ultimately settled
on testing only at $i=8,10,15$ by trial and error
on relatively small search regions.
Similarly, the threshold of $.002$ as an upper
bound for height ratios of interest was chosen
by trial and error --- small preliminary searches
suggested there were plenty of interesting points
with ratio below $.002$, but simply too many larger
than that.
\end{remark}


{\bf Acknowledgements.}
Most of the results of this paper came from an undergraduate
research project at Amherst College in Summer 2011.
Supplemental computations
were implemented by author T.H.\ in Spring~2012.
Authors R.B., T.H., and C.W.\ gratefully acknowledge the support of
NSF grant DMS-0600878 for this research.
Authors R.C.\ and Y.K.\ gratefully acknowledge the support of
Amherst College's Dean of Faculty student support funds.
The authors thank David Cox, Ben Hutz, and Masato Kuwata
for helpful discussions.
The authors also thank Amherst College,
especially Andy Anderson and Steffen Plotner,
for the use of and assistance with
the College's high-performance
computing cluster, on which we ran our computations.
Finally, we thank Ben Hutz, Joseph Silverman,
and the referee for their helpful comments
in improving the original draft of this paper.

\end{document}